\documentclass[12pt, reqno]{amsart}
\usepackage{amsmath, amsthm, amscd, amsfonts, amssymb, graphicx, color}
\usepackage[bookmarksnumbered, colorlinks, plainpages]{hyperref}

\newtheorem{theorem}{Theorem}[section]
\newtheorem{lemma}[theorem]{Lemma}
\newtheorem{proposition}[theorem]{Proposition}
\newtheorem{corollary}[theorem]{Corollary}
\theoremstyle{definition}

\newtheorem{example}[theorem]{Example}

\theoremstyle{remark}

\numberwithin{equation}{section}

\begin{document}

\title[Completeness properties]{Completeness properties of the space of quasicontinuous functions }

\author[\v Lubica Hol\'a]{\v Lubica Hol\'a}

\newcommand{\acr}{\newline\indent}

\address{\llap Academy of Sciences, Institute of Mathematics \acr \v
	Stef\'anikova 49,
	81473 Bratislava,
	\acr Slovakia}

\email{\textcolor[rgb]{0.00,0.00,0.84}{hola@mat.savba.sk}}





\bigskip

\bigskip

\bigskip

\bigskip

\begin{abstract}
	Quasicontinuous functions have found applications in many areas of mathematics. We study completeness properties of the space of quasicontinuous functions equipped with the topology of pointwise convergence. Let $X$ be a Hausdorff topological space, $Q(X)$ be the space of quasicontinuous real-valued functions and $\tau_p$ be the topology of the pointwise convergence. For $(Q(X),\tau_p)$ complete metrizability, Polishness and \v Cech-completeness are equivalent. If $(Q(X),\tau_p)$ is completely  metrizable, then $X$ is countable and the set $I(X)$ of isolated points of $X$ is dense in $X$. If $X$ is first countable, then $(Q(X),\tau_p)$ is completely  metrizable if and only if $X$ is countable and $I(X)$ is dense in $X$.
\end{abstract}

\maketitle

\section{Introduction}

\bigskip
Quasicontinuity is a classical notion. Quasicontinuous functions were introduced by Kempisty in 1932 in \cite{Kem}. However  as far as we know the first mention of the condition of quasicontinuity can be found in the paper of R. Baire \cite{Ba} in the study of continuity points of separately continuous functions from $\Bbb R^2$ into $\Bbb R$. Quasicontinuous functions are very important in many areas of mathematics,  they have been intensively studied in the
literature \cite{Ho1, Ho2, Hol, HH9, HH10, Na, O}, see also the recent book Usco and quasicontinuous mappings \cite{HHM} and references therein.
Quasicontinuous functions are selections of minimal usco and minimal cusco maps \cite{HH1, HH3, HH6}, they found applications  in the study of topological groups \cite{Bou, Mo1, Mo3}, in the study of dynamical systems \cite{CFM}, etc.  The quasicontinuity is also used in the study of CHART groups \cite{Mo2}.

In this paper we study the topology of the pointwise convergence on the space of quasicontinuous functions, mainly completeness properties. We continue our research from \cite{HH6}. Let $X$ be a Hausdorff topological space, $Q(X)$ be the space of quasicontinuous real-valued functions and $\tau_p$ be the topology of the pointwise convergence. For $(Q(X),\tau_p)$ complete metrizability, Polishness and \v Cech-completeness are equivalent. If $(Q(X),\tau_p)$ is completely  metrizable, then $X$ is countable and the set $I(X)$ of isolated points of $X$ is dense in $X$. If $X$ is first countable, then $(Q(X),\tau_p)$ is completely  metrizable if and only if $X$ is countable and $I(X)$ is dense in $X$.

\section{Preliminaries}

\bigskip

In what follows let $X, Y$ be Hausdorff topological spaces, $\Bbb N$ be the set of positive integers, $\omega$ be the set of
non-negative integers  and $\Bbb{R}$ be the space of real numbers with the usual metric.  The symbol $\overline A$ will stand for the closure of the set $A$ in a topological space.

A function $f: X \to Y$ is \emph{quasicontinuous} \cite{Ne} at $x \in X$ if for every open set $V \subset Y, f(x) \in V$ and every open set $U\subset X$, $x \in U$ there is a nonempty open set $W \subset U$ such that $f(W) \subset V$. If $f$ is quasicontinuous at every point of $X$, we say that $f$ is quasicontinuous.
We say that a subset of $X$ is quasi-open (or semi-open) \cite{Ne} if it is contained in the closure of its interior. Then a function $f: X \to Y$ is quasicontinuous if and only if $f^{-1}(V)$ is quasi-open for every open set $V\subset Y$ \cite{Ne}.

Let $X$ be a topological space and $(Y,d)$ be a metric space. A function $f: X \to Y$ is \emph{cliquish} \cite{Th} at $x_0$ if for every $\varepsilon > 0$ and every open neighbourhood $U$ of $x_0$, there is a nonempty open subset $V$ of $U$ such that for all $x_1, x_2 \in V$, $d(f(x_1),f(x_2)) < \varepsilon$. The function $f$ is cliquish if it is cliquish at every point $x \in X$.

For a topological space $X$ denote by $\Bbb R^X$ the set of all functions from $X$ to $\Bbb R$, by $Q(X)$ the set of all quasicontinuous functions in $\Bbb R^X$, by $C(X)$ the space of all continuous functions in $\Bbb R^X$ and by $Cliq(X)$ the set of all cliquish functions in $\Bbb R^X$.

\bigskip

By $\mathfrak{F}$ we mean the family of all finite subsets of $X$. Denote by $\tau_p$  the topology of pointwise convergence on $\Bbb R^X$.  This topology is induced by the
uniformity $\frak U_p$  which has a base consisting of sets of the
form
$$W(A,\varepsilon )=\{(f ,g):\ \forall\ x\in A\ \ |f(x) - g(x)|<
\varepsilon \},$$

where $A \in\mathfrak{F}$ and $\varepsilon >0$. The general
$\tau_p$-basic neighborhood of $f\in \Bbb R^X$ will be denoted by
$W(f,A,\varepsilon )$, i.e. $W(f,A,\varepsilon
)=W(A,\varepsilon )[f]$.

The topology $\tau_p$ of the pointwise convergence on $\Bbb R^X$ is just the product topology on $\Bbb R^X$. If we write $\Bbb R^X$,
we will always mean that it is equipped with the product topology.

In what follows let $X$ be a Hausdorff nontrivial topological space, i.e., $X$ is at least countable.

In this paper, we are  mainly interested in completeness properties. A topological space $X$ is \v Cech-complete  \cite{En}, if $X$ is a Tychonoff space and it is a $G_\delta$ set in one (equivalently, in all) of its compactifications. Each completely metrizable space is Čech-complete.

A set $A\subseteq X$ is of \emph{first Baire category}, if it is a countable union of nowhere dense sets; otherwise it is of \emph{second Baire category}. A space $X$ is called \emph{Baire} space, if every nonempty open subset of $X$ is of second Baire category. There is also a characterization of Baire spaces via topological games.

The Choquet game $G(X)$ of a topological space $X$  \cite{Kech} is a game between two players $\alpha$ and $\beta$. The player $\beta$ starts a play by selection a nonempty open subset $U_0$ of $X$. Then $\alpha $ chooses a nonempty open subset $V_0$ of $U_0$. In  $n$th move, the player $\beta$  picks  a nonempty open set $U_{n}
\subset V_{n-1}$, where $V_{n-1}$ is the previous move of $\alpha$-player,  and $\alpha$ answers  by selecting a nonempty open set $V_n \subset U_n$.
The player $\alpha$ wins the play $(U_i, V_i)_{i \in \omega}$, if $\bigcap_{n \in \omega}  V_n  \neq \emptyset$. Otherwise  the player $\beta $ is said to have won the play. \\
We say that the player $\alpha$ has a winning strategy for the game $G(X)$ if there exists a strategy $s$, such that $\alpha $  wins all plays provided that he/she acts according to the strategy $s$. In this case, we say that $X$ is a \emph{Choquet space}.

The strong Choquet game $G^s(X)$ of a topological space $X$  \cite{Kech} is a game between two players $\alpha$ and $\beta$ similar to $G(X)$. The player $\beta$ starts a play by selection a pair $(x_0,U_0)$, where $U_0$ is an open subset of $X$ and $x_0 \in U_0$. Then $\alpha $ must play a nonempty open subset $V_0$ of $U_0$ with $x_0 \in U_0$.

In $n$th move, the player $\beta$  picks  a pair $(x_n,U_n)$, where $U_n$ is an open subset of $V_{n-1}$, the previous move of $\alpha$-player, and $x_n \in U_n$. Then $\alpha$ answers  by selecting an open set $V_n \subset U_n$ with $x_n \in V_n$.
The player $\alpha$ wins the play $(U_i, V_i)_{i \in \omega} $, if $\bigcap_{n \in \omega}  V_n  \neq \emptyset$. Otherwise  the player $\beta $ is said to have won the play. \\
We say that the player $\alpha$ has a winning strategy for the game $G^s(X)$ if there exists a strategy $s$, such that $\alpha $  wins all plays provided that he/she acts according to the strategy $s$. In this case, we say that $X$ is a \emph{strong Choquet space}.

In both games we can suppose that the players choose open sets from a fixed base of $X$.

\bigskip

\section{Quasicontinuous functions with the topology of pointwise convergence}

\bigskip

\bigskip

We extend Corollary 4.11 in \cite{HH6} using the notion of a $q$-space. It is a topological space such that for each point there is a sequence $\{U_n:\ n\in \omega\}$ of neighbourhoods of that point so that if $x_n\in U_n$ for each $n$ then $\{x_n:\ n\in \omega\}$ has a cluster point \cite{MN}.

\begin{proposition} \label{q} Let $X$ be a Hausdorff topological space. The following are equivalent:

(1) $(Q(X),\tau_p)$ is  metrizable;

(2) $(Q(X),\tau_p)$ has a countable base;

(3) $(Q(X),\tau_p)$ is  first countable;

(4) $(Q(X),\tau_p)$ is  a $q$-space;

(5) $X$ is countable.
\end{proposition}
\begin{proof} It is sufficient to prove that $(4) \Rightarrow (5)$ since $1 \Leftrightarrow (2) \Leftrightarrow (3) \Leftrightarrow (5)$ are proved in Corollary 4.11 in \cite{HH6}. Suppose that $(Q(X),\tau_p)$ is a $q$-space. Let $f$ be the zero function on $X$. By the assumption there is a sequence $\{W(f,K_n,\varepsilon_n):\ n\in\omega\}$, $K_n \in \mathfrak{F}$, $\varepsilon > 0$,  such that if $f_n\in W(f,K_n,\varepsilon_n)$ for each $n\in\omega$, then $\{f_n:\ n\in\omega\}$ has a cluster point in $(Q(X),\tau_p)$. We claim  that $X=\bigcup\{K_n:n\in \omega\}$. Suppose that there is $x\in X\setminus\bigcup\{K_n:n\in\omega\}$.
For each $n\in\omega$, $K_n \in \mathfrak{F}$, thus there is an open set $U_n \subset X$ such that $\overline{U_n} \cap K_n = \emptyset$.
Let $f_n$ be such that $f_n(z) = n$ if $z \in \overline{U_n}$ and $f_n(z) = 0$ otherwise. Then $f_n$ is a quasicontinuous function and
 $f_n\in W(f,K_n,\varepsilon_n)$ for every $n\in\omega$, however $\{f_n: n \in \omega\}$ cannot have a cluster point in $(Q(X),\tau_p)$.
\end{proof}

\begin{corollary} Let $X$ be a Hausdorff topological space. The following are equivalent:

(1) $(Q(X),\tau_p)$ is  Polish;

(2) $(Q(X),\tau_p)$ is  completely metrizable;

(3) $(Q(X),\tau_p)$ is  \v Cech-complete.
\end{corollary}
\begin{proof} Using Proposition \ref{q} $(1) \Leftrightarrow (2)$. It is sufficient to prove that $(3) \Rightarrow (2)$. It is known that \v Cech-complete space is a $q$-space. Thus by Proposition \ref{q} $(Q(X),\tau_p)$ is  \v Cech-complete and metrizable, thus it is completely metrizable.
\end{proof}

\bigskip

\begin{lemma} (\cite{HH6}) \label{dense} Let $X$ be a Hausdorff topological space. Let $x_1, x_2, ..., x_n$ be points in $X$ and $r_1, r_2, ..., r_n \in \Bbb R$. Then there exists a quasicontinuous function $f: X \to \Bbb R$ such that $f(x_i) = r_i$ for any $i \in \{1, 2, ..., n\}$.
\end{lemma}

\begin{theorem} (\cite{HH6}) \label{dense1} Let $X$ be a Hausdorff topological space. $Q(X)$ is dense in $\Bbb R^X$.

\end{theorem}

\bigskip

\begin{theorem} Let $X$ be a Hausdorff topological space.  $(Q(X),\frak U_p)$ is complete if and only if $X$ is discrete.
\end{theorem}
\begin{proof} If $X$ is discrete, then $Q(X) = \Bbb R^X$ and we are done, since $(\Bbb R^X,\frak U_p)$ is complete \cite{Ke}. Suppose now that $(Q(X),\frak U_p)$ is complete. We will prove that every $f \in \Bbb R^X$ is quasicontinuous. Let $f \in \Bbb R^X$. Let $A \in \mathfrak{F}$. Consider $f|A$. By Lemma \ref{dense} there is a quasicontinuous extension $f_A: X \to \Bbb R$ of $f|A$. Let $\mathfrak{F}$ be directed by the inclusion. The net $\{f_A: A \in \mathfrak{F}\}$ is Cauchy in $(Q(X),\frak U_p)$. Consider $W(B,\epsilon) \in \frak U_p$. Let $C, D $ be in $\mathfrak{F}$ such that  $B \subset C$ and $B \subset D$. Then $(f_C,f_D) \in W(B,\epsilon)$. Since $(Q(X),\frak U_p)$ is complete, the net $\{f_A: A \in \mathfrak{F}\}$ pointwise converges to a quasicontinuous function. Since the net $\{f_A: A \in \mathfrak{F}\}$  pointwise converges to $f$, $f$ has to be quasicontinuous.
Suppose there is $x_0 \in X$ which is not isolated. Then the function $g: X \to \Bbb R$ defined by $g(x_0) = 0$ and $g(x) = 1$ otherwise, is not quasicontinuous.

\end{proof}

\bigskip

In what follows denote $I(X)$ the set of all isolated points of $X$.

\bigskip



\begin{theorem} \label{cliq} Let $X$ be a  Hausdorff topological space.
 If $(Q(X),\tau_p)$ is completely metrizable, then $X$ is countable and $I(X)$ is dense in $X$.
\end{theorem}
\begin{proof} Suppose that  $(Q(X),\tau_p)$ is completely metrizable.  By Proposition \ref{q} the metrizability of $(Q(X),\tau_p)$ implies that $X$ is countable. We prove that every  function in $\Bbb R^X$ has to be cliquish. We use an idea from \cite{A}. Let $g \in \Bbb R^X$. Define the map $\psi: \Bbb R^X \to \Bbb R^X$ as follows: $\psi(f) = f + g$ for all $f \in \Bbb R^X$.
Clearly, $\psi$ is a homeomorphism of $\Bbb R^X$ onto itself. By Theorem \ref{dense1} $Q(X)$ is dense in $\Bbb R^X$  and $Q(X)$ is $G_\delta$ in $\Bbb R^X$. Hence the set $\psi(Q(X))$ is dense in $\Bbb R^X$ and $G_\delta$ in $\Bbb R^X$. Thus $Q(X) \cap \psi(Q(X))$ is the intersection of a countable family of open dense sets in $\Bbb R^X$. Since $\Bbb R^X$ is a Baire space, $Q(X) \cap \psi(Q(X)) \ne \emptyset$. Let $h \in Q(X) \cap \psi(Q(X))$.
Then $h = f + g$ for some $f \in Q(X)$ and $h \in Q(X)$. Thus $g = h - f$. It is easy to verify that $g$ must be cliquish. Suppose $X \ne \overline{I(X)}$. Put $L = X \setminus \overline{I(X)}$. Enumerate $X$ as $\{x_n: n \in \omega\}$ and define $f: X \to \Bbb R$ as $f(x_n) = n$. Since every nonempty open subset $U$ of $L$ must be infinite, the diameter $diam(f(U)) = \infty$, thus $f$ cannot be cliquish, a contradiction.

\end{proof}

\bigskip

\begin{proposition} Let $X$ be a countable discrete topological space. Then $(Q(X),\tau_p)$ is completely metrizable.
\end{proposition}
\begin{proof} $Q(X) = \Bbb R^X$ and $\Bbb R^X$ equipped with the product topology is completely metrizable.

\end{proof}

\begin{theorem} \label{choquet} Let $X$ be a  first countable Hausdorff topological space with only one non isolated point.
Then  $(Q(X),\tau_p)$ is strong Choquet.
\end{theorem}
\begin{proof} Let $x_0$ be a non isolated point in $X$ and $\{G_n(x_0): n \in \omega\}$ be a base of decreasing neighbourhoods of $x_0$. We will define a winning strategy $s$ for the player $\alpha$ in $G^s(Q(X))$ as follows: given the initial step $(f_0,U_0)$ of $\beta$, where $U_0=W(f_0,B_0,\varepsilon_0)$ for some finite set $B_0 \subset X$ and $\varepsilon_0>0$, put $\delta_0=\varepsilon_0$ and choose $x^{f_0} \in G_0(x_0)$ such that $|f_0(x_0) - f_0(x^{f_0})| < \delta_0$. Put $A_0 = B_0 \cup \{x_0, x^{f_0}\}$ and define $s((f_0,U_0)) = V_0 = W(f_0,A_0,\delta_0)$.

Inductively,  given $n\ge 1$,  if  for all $i\le n$, $(f_i,U_i)$ have been chosen by $\beta$, where $U_i=W(f_i,B_i,\varepsilon_i)$
for some finite set $B_i\subset X$, $\varepsilon_i>0$, let $\delta_n=\min\{\delta_{n-1}/2,\varepsilon_n\}$, choose a $x^{f_n} \in G_n(x_0)$ such that $|f_n(x_0) - f_n(x^{f_n})| < \delta_n$. Put $A_n = A_{n-1} \cup B_n \cup \{x^{f_n}\}$ and define $s(f_0,U_0), V_0, ..., V_{n-1}, (f_n,U_n)) = V_n = W(f_n,A_n,\delta_n)$.

We will prove that $s$ is the winning strategy for $\alpha$.

Consider a  run  $(f_0,U_0),V_0,\dots,(f_n,U_n),V_n,\dots$ of $Ch^s(Q(X))$ compatible
with $s$. Put $L=\bigcup_{n\in\omega} A_n = \bigcup_{n\in\omega} B_n$.
For every $x\in L$, there is $n_0\in\omega$  with $x\in A_{n_0}$, and for all $n>m\ge k\ge  n_0$ we have $f_n\in W(f_m,A_m,\delta_m)$, so
$|f_n(x) - f_m(x)|<\delta_{k}$. It follows that $(f_n(x))_{n\in\omega}$ is Cauchy, since $\delta_k\to 0$, so it converges to some $f(x) \in \Bbb R$, moreover, $|f(x) - f_n(x)|\le\delta_{k}<\delta_{n}$. It is easy to verify that $f \in W(f_n|L,A_n,\delta_n)$ for every $n \in \omega$.

Now we prove that $f$ is quasicontinuous at $x_0$.  Let $\varepsilon > 0$ and $G_n(x_0)$ be an open neighbourhood of $x_0$.
Let $m > n$ be such that $\delta_m < \varepsilon/5$. Then $|f_m(x_0) - f_m(x^{f_m})| < \delta_m$ and $x^{f_m} \in G_m(x_0) \subset G_n(x_0)$. We show that $|f(x_0) - f(x^{f_m})| < \varepsilon$. Let $k > m$ be such that

\centerline{$|f_k(x_0) - f(x_0)| < \delta_m$ and $|f_k(x^{f_m}) - f(x^{f_m})| < \delta_m$. Then}

\centerline{$|f(x_0) - f(x^{f_m})| \le |f(x_0) - f_k(x_0)| + |f_k(x_0) - f(x^{f_m})| \le $}

\centerline{$ |f(x_0) - f_k(x_0)|  + |f_k(x_0) - f_k(x^{f_m})| + |f_k(x^{f_m}) - f(x^{f_m})| $}

\centerline{$\le 2\delta_m +  |f_k(x_0) - f_k(x^{f_m})|$.}

\bigskip

 Since $\{x_0, x^{f_m}\} \subset  A_m$ and $f_k \in W(f_m,A_m,\delta_m)$,we have  $|f_k(x_0) - f_k(x^{f_m})| \le |f_k(x_0) - f_m(x_0)| + |f_m(x_0) - f_k(x^{f_m})| \le |f_k(x_0) - f_m(x_0)|  +
 |f_m(x_0) - f_m(x^{f_m})| + |f_m(x^{f_m}) - f_k(x^{f_m})| \le 3\delta_m$. Thus $|f(x_0) - f(x^{f_m})| < \varepsilon$. Define the function $g: X \to \Bbb R$ as $g(x) = f(x)$, if $x \in L$ and $g(x) = 1$ otherwise. Then $g$ is a quasicontinuous function and $g \in \bigcap_{n \in \omega} W(f_n,A_n,\delta_n)$. So $s$ is the winning strategy for $\alpha$ in $Ch^s(Q(X))$.

\end{proof}

The following example shows that the assumption of the first countability in Theorem \ref{choquet} is essential.

\begin{example} Let $\omega$ be equipped with the discrete topology and $\beta \omega$  be the \v Cech-Stone compactification of $\omega$. Choose $q \in \beta \omega \setminus \omega$. Let $X = \omega \cup \{q\}$ and $X$ have the topology inherited from $\beta \omega$.  We show that every quasicontinuous function $f: X \to \Bbb R$ is continuous. Of course, $f$ is continuous at every point from $\omega$. Suppose that $f$ is not continuous at $q$. There is $\epsilon > 0$ such that in every neighbourhood $O$ of $q$ there is $x \in O$ with $f(x) \notin (f(q) - \epsilon,f(q) + \epsilon)$. Put $C_q = \{x \in \omega: f(x) \notin (f(q) - \epsilon,f(q) + \epsilon)\}$ and $Q_q = \{x \in \omega: f(x) \in (f(q) - \epsilon,f(q) + \epsilon)\}$. The sets $C_q$ and $Q_q$ are pairwise disjoint infinite closed sets in $\omega$, thus their closures in $\beta \omega$ are also disjoint.
Since $q$ has to be in the closure of $Q_q$, it cannot be in the closure of $C_q$, a contradiction.
It was shown in \cite{LM} that $C(X)$ with the  topology of pointwise convergence is Baire but not Choquet. Thus $(Q(X),\tau_p)$ cannot be strong Choquet, since every strong Choquet space is Choquet \cite{Kech}.

\end{example}

\begin{theorem} \label{ja} Let $X$ be a  first countable Hausdorff topological space with $I(X)$ dense and with  countably many non isolated points.
Then  $(Q(X),\tau_p)$ is strong Choquet.
\end{theorem}
\begin{proof} We will use an idea from Theorem \ref{choquet}. Enumerate as $\{x_n: n \in \omega\}$  non isolated points in $X$ and let $\{G_n(x_i): n \in \omega\}$ be a base of decreasing neighbourhoods of $x_i$ for every $i \in \omega$.  We will define a winning strategy $s$ for the player $\alpha$ in $G^s(Q(X))$ as follows: given the initial step $(f_0,U_0)$ of $\beta$, where $U_0=W(f_0,B_0,\varepsilon_0)$ for some finite set $B_0 \subset X$ and $\varepsilon_0>0$, put $\delta_0=\varepsilon_0$ and choose $x^{f_0} \in G_0(x_0)\cap I(X)$ such that $|f_0(x_0) - f_0(x^{f_0})| < \delta_0$. Put $A_0 = B_0 \cup \{x_0, x^{f_0}\}$ and define $s((f_0,U_0)) = V_0 = W(f_0,A_0,\delta_0)$.

Inductively,  given $n\ge 1$,  if  for all $i\le n$, $(f_i,U_i)$ have been chosen by $\beta$, where $U_i=W(f_i,B_i,\varepsilon_i)$
for some finite set $B_i\subset X$, $\varepsilon_i>0$, put $\delta_n=\min\{\delta_{n-1}/2,\varepsilon_n\}$. For every $i \le n$  choose  $x_i^{f_n} \in G_n(x_i)\cap I(X)$ such that $|f_n(x_i) - f_n(x_i^{f_n})| < \delta_n$. Put $A_n = A_{n-1} \cup B_n \cup \{x_n\} \cup \{x_i^{f_n}: i \le n\}$ and define $s(f_0,U_0), V_0, ..., V_{n-1}, (f_n,U_n)) = V_n = W(f_n,A_n,\delta_n)$.

We will prove that $s$ is the winning strategy for $\alpha$.

Consider a  run  $(f_0,U_0),V_0,\dots,(f_n,U_n),V_n,\dots$ of $Ch^s(Q(X))$ compatible
with $s$. Put $L=\bigcup_{n\in\omega} A_n = \bigcup_{n\in\omega} B_n$.
For every $x\in L$, there is $n_0\in\omega$  with $x\in A_{n_0}$, and for all $n>m\ge k\ge  n_0$ we have $f_n\in W(f_m,A_m,\delta_m)$, so
$|f_n(x) - f_m(x)|<\delta_{k}$. It follows that $(f_n(x))_{n\in\omega}$ is Cauchy, since $\delta_k\to 0$, so it converges to some $f(x) \in \Bbb R$, moreover, $|f(x) - f_n(x))\le\delta_{k}<\delta_{n}$. It is easy to verify that $f \in W(f_n|L,A_n,\delta_n)$ for every $n \in \omega$.

Now we prove that $f$ is quasicontinuous.  It is sufficient to prove that $f$ is quasicontinuous at every  $x_i$, $i \in \omega$, since other points are isolated. Let $i \in \omega$.  Let $\varepsilon > 0$ and $G_n(x_i)$ be an open neighbourhood of $x_i$.
Let $m > n$, $m > i$ be such that $\delta_m < \varepsilon/5$. Then $|f_m(x_i) - f_m(x_i^{f_m})| < \delta_m$ and $x_i^{f_m} \in G_m(x_i) \subset G_n(x_i)$. We show that $|f(x_i) - f(x_i^{f_m})| < \varepsilon$. Let $k > m$ be such that

\centerline{$|f_k(x_i) - f(x_i)| < \delta_m$ and $|f_k(x_i^{f_m}) - f(x_i^{f_m})| < \delta_m$. Then}

\centerline{$|f(x_i) - f(x_i^{f_m})| \le |f(x_i) - f_k(x_i)| + |f_k(x_i) - f(x_i^{f_m})| \le $}

\centerline{$ |f(x_i) - f_k(x_i)|  + |f_k(x_i) - f_k(x_i^{f_m})| + |f_k(x_i^{f_m}) - f(x_i^{f_m})| $}

\centerline{$\le 2\delta_m +  |f_k(x_i) - f_k(x_i^{f_m})|$.}

\bigskip

 Since $\{x_i, x_i^{f_m}\} \subset  A_m$ and $f_k \in W(f_m,A_m,\delta_m)$,we have  $|f_k(x_i) - f_k(x_i^{f_m})| \le |f_k(x_i) - f_m(x_i)| + |f_m(x_i) - f_k(x_i^{f_m})| \le |f_k(x_i) - f_m(x_i)|  +
 |f_m(x_i) - f_m(x_i^{f_m})| + |f_m(x_i^{f_m}) - f_k(x_i^{f_m})| \le 3\delta_m$. Thus $|f(x_i) - f(x_i^{f_m})| < \varepsilon$. Define the function $g: X \to \Bbb R$ as $g(x) = f(x)$, if $x \in L$ and $g(x) = 1$ otherwise. Then $g$ is quasicontinuous  and $g \in \bigcap_{n \in \omega} W(f_n,A_n,\delta_n)$. So $s$ is the winning strategy for $\alpha$ in $Ch^s(Q(X))$.

\end{proof}

\begin{proposition} \label{discrete} Let $X$ be a discrete topological space. Then $(Q(X),\tau_p)$ is strong Choquet.
\end{proposition}
\begin{proof} $Q(X) = \Bbb R^X$ and $\Bbb R^X$ equipped with the product topology is strong Choquet.

\end{proof}

\begin{theorem} \label{ja1} Let $X$ be a  first countable Hausdorff topological space.
$(Q(X),\tau_p)$ is completely metrizable if and only if $X$ is countable and $I(X)$ is dense in $X$.
\end{theorem}
\begin{proof} By Theorem \ref{cliq} if $X$ is a  Hausdorff topological space and
$(Q(X),\tau_p)$ is completely metrizable,  $X$ is countable and $I(X)$ is dense in $X$.
Suppose now that $X$ is a  first countable Hausdorff topological space, which is countable and $I(X)$ is dense in $X$. Since $X$ is countable, by Proposition \ref{q} $(Q(X),\tau_p)$ is metrizable.
To prove that it is also completely metrizable, it is sufficient to prove that $(Q(X),\tau_p)$ is strong Choquet \cite[Theorem 8.7]{Cho}. By Theorem \ref{ja} $(Q(X),\tau_p)$ is strong Choquet.
\end{proof}

\bigskip

In the proof of the next result we will use the notion of a resolvable topological space. A topological space is resolvable if it  can be expressed as the union of two disjoint dense subsets.  This notion was introduced by E. Hewitt in 1943 \cite{He}.

\begin{theorem} Let $X$ be a Hausdorff $k$-space. If $(Cliq(X),\frak U_p)$ is complete, then $I(X)$ is dense in $X$.
\end{theorem}
\begin{proof} Suppose that $(Cliq(X),\frak U_p)$ is complete. We will prove that every $f \in \Bbb R^X$ is cliquish. Let $f \in \Bbb R^X$. Let $A \in \mathfrak{F}$. Consider $f|A$. By Lemma \ref{dense} there is a quasicontinuous extension $f_A: X \to \Bbb R$ of $f|A$. Thus $f_A$ is cliquish for every $A \in \mathfrak{F}$. Let $\mathfrak{F}$ be directed by the inclusion. The net $\{f_A: A \in \mathfrak{F}\}$ is Cauchy in $(Q(X),\frak U_p)$. Consider $W(B,\epsilon) \in \frak U_p$. Let $C, D $ be in $\mathfrak{F}$ such that  $B \subset C$ and $B \subset D$. Then $(f_C,f_D) \in W(B,\epsilon)$. Since $(Cliq(X),\frak U_p)$ is complete, the net $\{f_A: A \in \mathfrak{F}\}$ pointwise converges to a cliquish function. Since the net $\{f_A: A \in \mathfrak{F}\}$  pointwise converges to $f$, $f$ has to be cliquish. We claim that $I(X)$ has to be dense in $X$.
Suppose $X \ne \overline{I(X)}$. Put $L = X \setminus \overline{I(X)}$. Since $L$ is open, $L$ is also a $k$-space. It is easy to verify that $L$ is dense-in-itself. Every dense-in-itself $k$-space is resolvable \cite{Ve}. Let $U_1, U_2$ be two disjoint dense subsets of $L$ such that $L = U_1 \cup U_2$. Define the function $g$ as follows: $g(x) = 1$ if $x \in U_1$ and $g(x) = 0$ otherwise. Then $g$ cannot be cliquish, since diameter $diam(g(V)) = 1$ for every nonempty open subset $V$ of $L$.

\end{proof}

\begin{proposition} Let $X$ be a Hausdorff topological space.  If $I(X)$ is dense in $X$, then $(Cliq(X),\frak U_p)$ is complete.
\end{proposition}
\begin{proof} If $I(X)$ is dense in $X$, then every function from $X$ to $\Bbb R$ is cliquish. Thus $Cliq(X) = \Bbb R^X$ and $(\Bbb R^X,\frak U_p)$ is complete \cite{Ke}.

\end{proof}

\begin{corollary} Let $X$ be a Hausdorff $k$-space. Then $(Cliq(X),\frak U_p)$ is complete, if and only if $I(X)$ is dense in $X$.
\end{corollary}

\begin{theorem} Let $X$ be a Hausdorff topological space. $(Cliq(X),\tau_p)$ is completely metrizable if and only if $X$ is countable and $I(X)$ is dense.
\end{theorem}
\begin{proof} If $X$ is countable, then $\Bbb R^X$ is completely metrizable. If $I(X)$ is dense in $X$, then every function from $X$ to $\Bbb R$ is cliquish. Thus $(Cliq(X),\tau_p) = \Bbb R^X$ is completely metrizable.

Suppose that $(Cliq(X),\tau_p)$ is completely metrizable. The metrizability of $(Cliq(X),\tau_p)$ implies that $(Q(X),\tau_p)$ is metrizable and thus by Proposition \ref{q} $X$ is countable. We prove that every  function in $\Bbb R^X$ has to be cliquish. We use a similar idea as in Theorem \ref{cliq}. Let $g \in \Bbb R^X$. Define the map $\psi: \Bbb R^X \to \Bbb R^X$ as follows: $\psi(f) = f + g$ for all $f \in \Bbb R^X$.
Clearly, $\psi$ is a homeomorphism of $\Bbb R^X$ onto itself. By Theorem \ref{dense1} $Q(X)$ is dense in $\Bbb R^X$. Thus also $Cliq(X)$ is dense in $\Bbb R^X$, since $Q(X) \subset Cliq(X)$. $Cliq(X)$ is $G_\delta$ in $\Bbb R^X$, since $(Cliq(X),\tau_p)$ is completely metrizable. Hence the set $\psi(Cliq(X))$ is dense in $\Bbb R^X$ and $G_\delta$ in $\Bbb R^X$. Thus $Cliq(X) \cap \psi(Cliq(X))$ is the intersection of a countable family of open dense sets in $\Bbb R^X$. Since $\Bbb R^X$ is a Baire space, $Cliq(X) \cap \psi(Cliq(X)) \ne \emptyset$. Let $h \in Q(X) \cap \psi(Q(X))$.
Then $h = f + g$ for some $f \in Cliq(X)$ and $h \in CliqX)$. Thus $g = h - f$. It is easy to verify that $g$ must be cliquish. Suppose $X \ne \overline{I(X)}$. Put $L = X \setminus \overline{I(X)}$. Enumerate $X$ as $\{x_n: n \in \omega\}$ and define $f: X \to \Bbb R$ as $f(x_n) = n$. For every nonempty open subset $U$ of $L$, the diameter $diam(f(U)) = \infty$, thus $f$ cannot be cliquish, a contradiction.

\end{proof}

\begin{corollary} Let $X$ be a  Hausdorff topological space.
$(Cliq(X),\tau_p)$  is Polish if and only if $(Cliq(X),\tau_p)$  is completely metrizable.
\end{corollary}
\begin{proof} If $(Cliq(X),\tau_p)$  is Polish, then it is completely metrizable. Suppose that $(Cliq(X),\tau_p)$  is completely metrizable. The metrizability of $(Cliq(X),\tau_p)$ implies the metrizability of $(Q(X),\tau_p)$ and it implies that $X$ is countable by Proposition \ref{q}. If $X$ is countable, then $\Bbb R^X$ is second countable, thus also $(Cliq(X),\tau_p)$  is second countable, since $Cliq(X) \subset \Bbb R^X$.
\end{proof}

\begin{proposition} Let $X$ be a Hausdorff topological space.  If $I(X)$ is dense in $X$, $(Cliq(X),\tau_p)$ is strong Choquet space.
\end{proposition}
\begin{proof} If $I(X)$ is dense in $X$, then every function from $X$ to $\Bbb R$ is cliquish. Thus $(Cliq(X),\tau_p) = \Bbb R^X$ is strong Choquet.

\end{proof}

\bigskip


Let $X$ and $Y$ be topological spaces. Denote  by $Q(X,Y)$ the set of all quasicontinuous functions from $X$.

\begin{lemma} \label{sum}
If $X= \bigoplus_{i\in I} X_i$, a topological sum of  topological spaces $X_i$, $i \in I$,  and $(Y,d)$ is a metric space, then $(Q(X,Y),\tau_{p})$ is homeomorphic to the product $\prod_{i \in I} (Q(X_i,Y),\tau_{p})$.
\end{lemma}
\begin{proof} Define the mapping $\Psi: (Q(X,Y),\tau_{p}) \to \prod_{i \in I} (Q(X_i,Y),\tau_{p})$ as follows: $\Psi(f) = (f|X_i)_{i \in I}$ for $f \in Q(X,Y)$. For every $i \in I$ the function $f|X_i \in Q(X_i,Y)$, since $X_i$ is an open set in $X$. It is easy to verify that $\Psi$ is a bijection. If $(f_i)_{i \in I} \in \prod_{i \in I} Q(X_i,Y)$,
then $\Psi^{-1}((f_i)_{i \in I})$ is the following function: $\Psi^{-1}((f_i)_{i \in I})(x) = f_i(x)$ if $x \in X_i$. Of course, $\Psi^{-1}((f_i)_{i \in I})$ is a quasicontinuous function from $X$ to $Y$. Now we prove that $\Psi$ is homeomorphisms.
Suppose that the net $\{f_\sigma:\sigma \in \Sigma\}$ converges to $f \in (Q(X,Y),\tau_{p})$. Let $O $ be an open neighbourhood of $\Psi(f)$ in  $\prod_{i \in I} (Q(X_i,Y),\tau_{p})$. There are a finite subset $J \subset I$, finite sets $K_j \subset X_j$, $j \in J$ and positive $\epsilon_j$, $j \in J$ such that

\bigskip

\centerline{$\prod_{j \in j} W(f|X_j,K_j,\epsilon_j) \times \prod_{i \in I \setminus J} Q(X_i,Y)   \subset O.$}

\bigskip

Put $K = \cup_{j \in J} K_j$ and $\epsilon = min\{\epsilon_j: j \in J\}$. There is $\sigma_0 \in \Sigma$ such that for every $\sigma \ge \sigma_0$, $f_\sigma \in W(f,K,\epsilon)$. Thus for every $\sigma \ge \sigma_0$, $\Psi(f_\sigma) \in O$.

Suppose now that the net $\{F_\sigma:\sigma \in \Sigma\}$ converges to $F$ in $\prod_{i \in I} (Q(X_i,Y),\tau_{p})$. Then $F = (f_i)_{i \in I}$ and $F_\sigma = (f_{\sigma,i})_{i \in I}$ for every $\sigma \in \Sigma$. Let $G$ be an open neighbourhood of $\Psi^{-1}((f_i)_{i \in I})$. There are a finite
 set $K \subset X$ and $\epsilon > 0$ such that

\bigskip
\centerline{$W(\Psi^{-1}((f_i)_{i \in I}),K,\epsilon) \subset G.$}
\bigskip
There is a finite subset $J $ of $I$ such that $K = \cup_{j \in J} K \cap X_j$. There is $\sigma_0 \in \Sigma$ such that

\bigskip
\centerline{$(f_{\sigma,i})_{i \in I} \in \prod_{j \in j} W(f|X_j,K \cap X_j,\epsilon) \times \prod_{i \in I \setminus J} Q(X_i,Y)$ for every $\sigma \ge \sigma_0.$}
\bigskip

Thus for every $\sigma \ge \sigma_0$ $\Psi^{-1}((f_{\sigma,i})_{i \in I}) \in W(\Psi^{-1}((f_i)_{i \in I}),K,\epsilon) \subset G.$

\end{proof}

\bigskip

\begin{corollary} Let $X= \bigoplus_{i\in I} X_i$, be a topological sum of  topological spaces $X_i$, $i \in I$, where every $X_i$ is either a discrete topological space or a first countable Hausdorff topological space with $I(X_i)$ dense and with  countably many non isolated points in $X_i$. Then $(Q(X),\tau_p)$ is strong Choquet.
\end{corollary}
\begin{proof} By Proposition \ref{discrete} and Theorem \ref{ja} every $(Q(X_i),\tau_p)$ is strong Choquet. By Lemma \ref{sum} $(Q(X),\tau_{p})$ is homeomorphic to the product $\prod_{i \in I} (Q(X_i),\tau_{p})$. By \cite{Kech} any product of strong Choquet spaces is strong Choquet.
\end{proof}

\end{document}